\documentclass[a4paper, 11 pt]{amsart}

\usepackage{amsthm,amssymb,amsmath,amsfonts,mathrsfs,amscd}  
\usepackage[latin1]{inputenc}
\usepackage{enumitem}
\usepackage{verbatim,cite}
\usepackage{leftidx}
\usepackage[all,2cell]{xy}
\usepackage[margin=1in, a4paper]{geometry}
\usepackage{calligra} 
\usepackage{parskip}
\usepackage{url}

\DeclareMathAlphabet{\mathcalligra}{T1}{calligra}{m}{n} \DeclareFontShape{T1}{calligra}{m}{n}{<->s*[2.2]callig15}{} 

\numberwithin{equation}{subsection}
\newtheorem{theorem}[equation]{Theorem}

\newtheorem{lemma}[equation]{Lemma}
\newtheorem{proposition}[equation]{Proposition}

\theoremstyle{definition}

\theoremstyle{definition}
\newtheorem{definition}[equation]{Definition}
\newtheorem{remark}[equation]{Remark}

\newcommand{\fr}{\mathfrak}
\newcommand{\cl }{\mathcal}

\def \A {\mathbb{A}}
\def \Q{\mathbb{Q}}

\def \C {\mathbb{C}}
\def \T {\mathbb{T}}
\def \R {\mathbb{R}}

\def\Z{\mathbb Z}

\def \qpbar {\overline{\Q}_p}

\def\gp{\mathfrak{p}}

\def \cald {\mathcal{D}}
\def \cala {\mathcal{A}}

\def \calo {\mathcal{O}}

\def\scrx{\mathscr{X}}
\def\scra{\mathscr{A}}
\def\scrm{\mathscr{M}}
\def\scrw{\mathscr{W}}

\def\scrz{\mathscr{Z}}
\def\scrd{\mathscr{D}}

\def\Hom{\mathrm{Hom}}

\def\GL{\mathrm{GL}}

\def\Res{\mathrm{Res}}
\def\det{\mathrm{det}}
\def\dim{\mathrm{dim}}
\def\End{\mathrm{End}}

\def\WD{\mathrm{WD}}

\def\cl{\mathrm{cl}}

\def\Gal{\mathrm{Gal}}
\def\As{\mathrm{As}}
\def\Frob{\mathrm{Frob}}
\def\boldG{\mathbf{G}}

\title{$p$-adic Asai Transfer}
\author{Baskar Balasubramanyam}
\address{Department of Mathematics, Indian Institute of Science Education and Research Pune, Dr.\ Homi Bhabha Road, Pashan, Pune 411008.}
\email{baskar@iiserpune.ac.in}
\author{Dipramit Majumdar}
\address{Department of Mathematics, Indian Institute of Technology Madras, Chennai 600036}
\email{dipramit@iitm.ac.in}


\date{\today}

\begin{document}
\maketitle

\begin{abstract}
Let $K/\Q$ be a real quadratic field. Given an automorphic representation $\pi$ for $\GL_{2}/K$, let $\As^\pm (\pi)$ denote the plus/minus Asai transfer of $\pi$ to an automorphic representation for $\GL_{4}/\Q$. In this paper, we construct a rigid analytic map from the universal eigenvariety of $\GL_{2}/K$ to the universal eigenvariety of $\GL_{4}/\Q$, which at nice classical points interpolate this Asai transfer.
\end{abstract}


\section{Introduction}

Let $K$ be a real quadratic field. Let $k \ge 2$ be an integer. Let $f$ be a Hilbert modular form over $K$ of weight $k$ (i.e., of parallel weight $(k,k)$) and level $1$. Further assume that $f$ is an eigenform for all the Hecke operators with eigenvalues $c(\fr a, f)$, where $\fr a$ runs over integral ideals of $\calo_K$ (the ring of integers in $K$). The standard $L$-function for $f$ is constructed from $c (\fr a, f)$ as a Dirichlet series over $K$.

In \cite{Asai}, Asai introduced the following $L$-function which is constructed only from the Hecke eigenvalues for ideals that come from $\Q$.  This is now referred to as the Asai $L$-function. More specifically, the Asai $L$-function is defined as
$$
G(s,f) = \zeta (2s - 2k + 2) \sum_{m=1}^\infty \frac{c(m \calo_K, f)}{m^s}.
$$
A priori, this function converges for $s$ in a certain right half-plane. It is known that this $L$-function has an Euler product expansion, analytic continuation to all of $\C$ and a functional equation. 

Let $\pi$ be the automorphic representation associated to the form $f$. Then the $L$-function $G(s, f)$ is (up to a shift) a certain automorphic $L$-function, denoted by $L(s, \pi, \As^+)$, associated to $\pi$. The principle of Langlands functoriality suggests, in this case, that this automorphic $L$-function is the standard $L$-function of an automorphic representation for $\GL_4/\Q$. The details of this Asai transfer are recalled in \S \ref{sec:asai-transfer}. Fix a prime $p$ that is unramified in $K$. The aim of this article is to construct a $p$-adic version of this Asai transfer.

Specifically, when $p$ splits in $K$, we construct a rigid analytic map from the eigenvariety attached to $\GL_2/K$ to the eigenvariety attached to $\GL_4/\Q$ that interpolate the classical Asai transfer on a dense subset of classical points. But when $p$ is inert, we are only able to construct a map to an eigenvariety which can be viewed as a quotient of the the eigenvariety attached to $\GL_4/\Q$. 

Historically, the study of Langlands' functoriality in families of automorphic forms can be traced back to the work of Hida \cite{H}, where a $\Lambda$-adic Jacquet--Langlands' transfer is constructed between families of Hilbert modular forms and of quaternionic automorphic forms. In the setting of eigenvarieties, Chenevier \cite{Che} constructed $p$-adic Jacquet--Langlands transfer,  which at classical points interpolate the classical Jacquet--Langlands' transfer for $\GL_{2}$. Following Chenevier's method, other instances of $p$-adic Langlands' functoriality have been established, see \cite{New13, Whi14, Lud14, Maj16, Lud17, Hans}.

This article is organized as follows. In \S \ref{sec:asai-transfer}, we recall some of the basic properties of the Asai transfer. In \S \ref{sec:universal-eigenvarieties}, we recall the construction due to Hansen of universal eigenvarieties attached to certain connected reductive groups. In \S \ref{sec:relevant-eigenvarieties}, we explicitly describe the eigenvarieties that are relevant to our construction of the $p$-adic Asai transfer map. In \S \ref{sec:p-adic-asai-transfer}, we finally construct the $p$-adic Asai transfer map between eigenvarieties using a comparison theorem due to Hansen. 

\subsection*{Notations}
Throughout this article, $p$ will denote a fixed odd integer prime, and $K$ a real quadratic extension in which $p$ is unramified. The ring of adeles over $\Q$ will be denoted by $\A = \A_\Q$. We will denote by $\A_{f}$ and $\A_{\infty}$ the finite adeles and the adeles at infinity, respectively. For a set of places $S$ of $\Q$, we will denote by $\A_{S}$ the adeles supported at $S$ and by $\A^S$ the adeles supported away from $S$. We will denote by $\A_K = \A \otimes_\Q K$ the adeles over $K$. We similarly define $\A_{K, f}, \A_{K, \infty}, \A_K^S$ and $\A_{K,S}$ when $S$ is a set of places of $K$.

\subsection*{Acknowledgements}
The first author would like to acknowledge the support of SERB grant EMR/2016/000840 in this project. The second named author would like to acknowledge the support of NFIG grant of IIT Madras numbered MAT/16-17/839/NFIG/DIPR.

\medskip

\section{Asai transfer}\label{sec:asai-transfer}

Let $K = \Q(\sqrt{d})$ be a real quadratic field. We now recall the Langlands functorial transfer from automorphic representations for $\GL_2$ over $K$ to automorphic representations of $\GL_4$ over $\Q$ that is called the Asai transfer. 

\smallskip

\subsection{Map between $L$-groups}

Let $\boldG_1$ denote the algebraic group $\GL_2$ over $K$ and let $G_1 = \mathrm{Res}_{K/\Q} \boldG_1$ denote the Weil restriction of $\boldG_1$ from $K$ to $\Q$.
 The Langlands dual group for $G_1$ is defined as 
$$
{^L}G_1 = (\GL_{2}(\C) \times \GL_{2}(\C)) \rtimes \mathrm{Gal}(K/\Q)
$$
where the nontrivial element $c \in \Gal (K/\Q)$ acts on the tuple of matrices via permutation. Let $G_2$ denote the algebraic group $\GL_4$ over $\Q$. The $L$-group of $G_2$ is given by ${^L} G_2 = \GL_{4}(\C).$ Let $\As^\pm$ denote the following representation of ${^L} G_1$ acting on $\C^2 \otimes \C^2$ given by 
\begin{gather*}
\As^\pm [(A,B)] (x \otimes y) = Ax \otimes By, \textrm{ for } A, B \in \GL_2 (\C) \\
\As^\pm [c] (x \otimes y) = \pm (y \otimes x).
\end{gather*}
We view these representations as maps between the $L$-groups $\As^\pm: {^L} G_1 \to {^L} G_2$. The main results of \cite{Kris}  and \cite{Ram} show that the Asai transfer is automorphic. 

\smallskip
\subsection{The Asai motive}

We now recall some basic facts about Hilbert modular forms and the associated Asai motive. Let $\sigma_1, \sigma_2 : K \to \R$ denote the two embeddings of $K$ into $\R$. The weights of Hilbert modular forms are elements of the lattice $\Z[\sigma]:= \Z \sigma_1 + \Z \sigma_2$. Specifically, the weight of a Hilbert modular form will correspond to a pair $(n,v)$ with $n, v \in \Z[\sigma]$ such that $n+2v$ is parallel. By a parallel weight, we mean that $n+2v = m t$ where $t = \sigma_1 + \sigma_2$ and $m \in \Z$. Writing $n = n_1 \sigma_1 + n_2 \sigma_2$, this condition implies that the parity of $n_1$ and $n_2$ are the same. The parity assumption is necessary for the existence of Hilbert modular forms of a particular weight. Let $k = n + 2t$ and assume further that $k_i \ge 2$. 

Let $\mathfrak n$ be an integral ideal in $K$. Let $f$ be a Hilbert cusp form of weight $(n,v)$ and level $\mathfrak n$. Suppose that $f$ is a primitive eigenform, then the motive $M$ attached to $f$ is a pure simple rank $2$ motive defined over $K$. The Hodge types of the motive at $\sigma_i$ are $\{ (n_i + 1 + v_i, v_i), (v_i, n_i + 1 + v_i) \}$. If $\mathfrak a$ is an integral ideal in $K$, let  $c(\mathfrak a, f)$ denote the Hecke eigenvalue for the Hecke operator $T(\mathfrak a)$. Let ${^c} f$ be the Hilbert modular form whose Hecke eigenvalues are given by $c(\mathfrak a, {^c} f) = c(\mathfrak a^c, f)$. Let ${^c} M$ denote the motive associated to ${^c} f$. The Hodge type of this conjugate motive at $\sigma_i$ will be the Hodge type at $\sigma_i \circ c$. 

The motives $\As^\pm (M)$ associated to the Asai transfer will be pure simple rank $4$ motives defined over $\Q$. The Hodge types at the infinite place of $\Q$ are
\begin{gather*}
(n_1 + n_2 + v_1 + v_2 + 2, v_1 + v_2), \\
(n_1 + 1 + v_1 + v_2 , n_2 + 1 + v_1 + v_2), \\
(n_2 + 1 + v_1 + v_2 , n_1 + 1 + v_1 + v_2), \\
(v_1 + v_2, n_1 + n_2 + v_1 + v_2 + 2). 
\end{gather*}
Note that the weight of this motive is $n_1 + n_2 + 2(v_1 + v_2) + 2 = 2 m + 2$. Note also that when $n_1 = n_2$, the Asai motive has a middle (i.e., $(p,p)$) Hodge type and is hence not cohomological. Henceforth, we assume that $n_1 > n_2$.

\smallskip
\subsection{Weight of the Asai transfer} 

Let $G_1$ and $G_2$ be as above. A weight for $G_1$ is a tuple $\lambda = (\lambda_1, \lambda_2)$ where $\lambda_i \in \Z^2$. We say that the weight $\lambda$ is dominant if $\lambda_i = (a_i, b_i)$ with $a_i \ge  b_i$. We can relate these weights to the ones discussed above by taking $a_i = n_i + v_i$ and $b_i = v_i$.

Let $\pi$ be an automorphic representation for $G_1$. Let $\lambda$ be a dominant weight for $G_1$. We say that $\pi$ is cohomological of weight $\lambda$ if 
$$
H^* (\fr g_{1,\infty}, K_{1,\infty}; \pi_\infty \otimes \mathscr L_\lambda) \neq 0,
$$
where $\mathfrak g_{1,\infty}$ is the Lie algebra of $G_{1,\infty} = G_1 (\R)$ and $K_{1,\infty}$ is the maximal compact modulo the centre in $G_{1,\infty}$, and  $\mathscr L_\lambda$ is the highest weight representation associated to $\lambda$. In order for $\lambda$ to support cohomological automorphic representations, we require that $\lambda$ be pure, i.e., $a_1 + b_1 = a_2 + b_2$. 

A weight for $G_2$ is a tuple $\mu \in \Z^4$. We say that $\mu =  (\mu_1, \mu_2, \mu_3, \mu_4)$ is dominant if $\mu_1 \ge \mu_2 \ge \mu_3 \ge \mu_4$.  If $\Pi$ is an automorphic representation for $G_2$ and $\mu$ a dominant weight for $G_2$, we define the notion of $\Pi$ being cohomological of weight $\mu$ in a similar fashion. Similarly, for $\mu$ to support cohomological automorphic representations, we require that $\mu$ is pure, i.e., $\mu_1 + \mu_4 = \mu_2 + \mu_3$.

Let $\pi$ be the automorphic representation over $G_1 (\A)$ attached to the primitive eigenform $f$ of weight $(n,v)$. Then $\pi_\infty = \pi_1 \otimes \pi_2$, where $\pi_\infty$ is the representation at infinity and $\pi_i$ are  discrete series representations up to twists by powers of the determinant. The Langlands parameter of $\pi_i$ is given by 
$$
\tau (\pi_i ) = z^{\frac{1}{2} - v_i} \overline{z}^{- n_i - v_i - \frac{1}{2} } + z^{- n_i - v_i - \frac{1}{2} }  \overline{z}^{\frac{1}{2} - v_i}.
$$
One calculates that the Langlands parameter of the Asai transfer is
\begin{multline*}
\tau (\As (\pi)) =   z^{1 - v_1 -  v_2}  \overline{z}^{ - n_1 - n_2 - v_1 - v_2 - 1} + z^{ -n_1 - n_2 - v_1 - v_2 - 1} \overline{z}^{1 - v_1 - v_2} \\
 +  z^{ - n_2 - v_1 - v_2 } \overline{z}^{ -n_1 - v_1 - v_2} + z^{ - n_1 - v_1 - v_2 } \overline{z}^{- n_2 - v_1 - v_2}.
\end{multline*}
As the exponents in the Langlands parameter are not half-integers, $\As(\pi)$ is not cohomological. However, if we normalize the Asai transfer to be $\As(\pi) \otimes |\det|^{1/2}$, then the Langlands parameter becomes
\begin{multline*}
  z^{\frac{3}{2} - v_1 -  v_2}  \overline{z}^{ - n_1 - n_2 - v_1 - v_2 - \frac{1}{2}} + z^{ -n_1 - n_2 - v_1 - v_2 - \frac{1}{2}} \overline{z}^{\frac{3}{2} - v_1 - v_2} \\
 +  z^{\frac{1}{2} - n_2 - v_1 - v_2 } \overline{z}^{\frac{1}{2} -n_1 - v_1 - v_2} + z^{\frac{1}{2} - n_1 - v_1 - v_2 } \overline{z}^{\frac{1}{2} - n_2 - v_1 - v_2}.
\end{multline*}
One can verify that this representation is supported in cohomology of weight $\mu = (\mu_1, \mu_2, \mu_3, \mu_4) \in \Z^4$ which we now describe. The weight $\mu$ is pure (in the sense of Clozel) with purity weight $w = 2m-1$ , i.e., $\mu_1 + \mu_4 = w = \mu_2 + \mu_3$ (here $m=n_{1}+2v_{1}$). We also have 
\begin{gather*}
\mu_1 = \frac{n_1 + n_2}{2} + m -1 \\
\mu_2 = \frac{n_1 - n_2}{2} + m-1.
\end{gather*}
We easily calculate 
\begin{gather*}
\mu_3 = m - \frac{n_1 - n_2}{2} \\
\mu_4 =m - \frac{n_1 + n_2}{2} 
\end{gather*}
from the purity condition.

\smallskip
\subsection{Local Asai transfer}
In this section, we describe the local Asai transfer at almost all finite places in terms of Langlands' parameters. Let $k$ be any local field, let $W_k$ denote the Weil group for $k$.  Let $\ell$ be an integer prime that is unramified in $K$. We have two cases depending on whether $\ell$ splits or is inert in $K$.

First we assume that $\ell$ is split in $K$. Say, $\ell = \mathfrak l \mathfrak l^c$. Since $G_1 = \Res_{K/\Q} (\boldG_1)$, we have $(G_1)_{\ell} = G_1/\Q_\ell = (\boldG_1)_{\mathfrak l} \times (\boldG_1)_{ \mathfrak l^c}$. 
We know that ${^L} (\boldG_1)_{\mathfrak l} = {^L} (\boldG_1)_{\mathfrak l^c} = \GL_2 (\C)$ and  $^{L}(G_1)_\ell =  {^L} (\boldG_1)_{\mathfrak l} \times {^L} (\boldG_1)_{\mathfrak l^c}$. By an $L$-parameter for  $(\boldG_1)_{\mathfrak l}$, we mean a continuous morphism
$$
\varphi: W_{K_{\mathfrak l}} \to {^L} (\boldG_1)_{\mathfrak l}
$$
such that $\varphi (x) $ is semi-simple for all $x \in W_{K_{\mathfrak l}}$. Similarly take an $L$-parameter for $(\boldG_1)_{\mathfrak l^c}$ denoted by
$$
\varphi^c: W_{K_{\mathfrak l^c}} \to {^L} (\boldG_1)_{\mathfrak l^c}.
$$
Identifying $W_{K_{\mathfrak l}}$ with $W_{\Q_\ell}$ and using $\varphi$ and $\varphi^c$, we now construct an $L$-parameter for $(G_1)_\ell$ as
$$
\tilde \varphi = \varphi \times \varphi^c: W_{\Q_\ell}  \to {^L} (G_1)_\ell.
$$
Now let $\Frob_{\mathfrak l}, \Frob_{\mathfrak l^c}$ and $\Frob_{\ell}$ denote the Frobenius elements associated to $\mathfrak l, \mathfrak l^c$ and $\ell$ respectively. 

Suppose that the 
$$
\varphi(\Frob_{\mathfrak{l}}) = \bmatrix \alpha_{\fr l} & \\ & \beta_{\fr l} \endbmatrix \quad \mathrm{and} \quad \varphi^c(\Frob_{\mathfrak{l^c}}) = \bmatrix \alpha_{\fr l^c} & \\ & \beta_{\fr l^c}\endbmatrix,
$$ 
then we see that
$$
\tilde \varphi (\Frob_\ell) = \left(\bmatrix \alpha_{\fr l} & \\ & \beta_{\fr l} \endbmatrix, \bmatrix \alpha_{\fr l^c} & \\ & \beta_{\fr l^c} \endbmatrix \right).
$$
Applying the Asai transfer $\As^\pm$, we see that
$$
\As^\pm \tilde \varphi (\Frob_\ell) = \bmatrix \alpha_{\fr l} \alpha_{\fr l^c} & \\ & \alpha_{\fr l} \beta_{\fr l^c} \\ & & \beta_{\fr l} \alpha_{\fr l^c} \\ & & & \beta_{\fr l} \beta_{\fr l^c} \endbmatrix.
$$
The calculation above shows that the local Asai transfer for $f$ at a split place $\ell=\fr l \fr l^c$ is same as local Rankin--Selberg transfer for $f \times {^c}f$ at $\ell$.

Now we assume that $\ell$ is inert in $K$. Let $K_\ell$ denote the completion of $K$ at $\ell$. This is an unramified quadratic extension of $\Q_\ell$. In this case, ${^L} (G_1)_\ell = ({^L} (\boldG_1)_\ell \times {^L} (\boldG_1)_\ell ) \rtimes \Gal (K_\ell/\Q_\ell)$ where ${^L} (\boldG_1)_\ell = \GL_2 (\C)$. We also know that $W_{K_\ell}$ is an index $2$ subgroup of $W_{\Q_\ell}$. Let 
$$
\varphi: W_{K_\ell} \to {^L} (\boldG_1)_\ell
$$
be an $L$-parameter of $(\boldG_1)_\ell$. There is an extension of $\varphi$ to an $L$-parameter of $(G_1)_\ell$ constructed as follows. Pick any $j \in W_{\Q_\ell} \setminus W_{K_\ell}$. Define a map
$$
\tilde \varphi : W_{\Q_\ell} \to {^L} G_\ell
$$
by sending $x \in W_{K_\ell}$ to $(\varphi (x), \varphi (jx j^{-1})) \times 1$ and by sending $j$ to $(\mathrm{Id}, \varphi (j^2)) \times c$. Composing $\tilde \varphi$ with $\As^\pm$, we get the $L$-parameter for the Asai transfer.

Now take $j = \Frob_\ell \in W_{\Q_\ell} \setminus W_{K_\ell}$ and we take $\Frob^2_\ell \in W_{K_\ell}$ to be the Frobenius element for the ideal $\ell \calo_{K_\ell}$ over $K_\ell$. Suppose that 
$$
\varphi (\Frob_\ell^2) = \bmatrix \alpha_\ell & \\ & \beta_\ell \endbmatrix,
$$
then we see that 
$$
\tilde \varphi (\Frob_\ell) = \left( \mathrm{Id}, \bmatrix \alpha_\ell & \\ & \beta_\ell \endbmatrix \right) \times c.
$$
Applying the Asai transfer map, we get
$$
\As^\pm \tilde \varphi (\Frob_\ell) = \bmatrix \pm \alpha_\ell & \\ & \pm \beta_\ell \\ & & & \pm \alpha_\ell \\ & & \pm \beta_\ell & \endbmatrix.
$$
This last matrix is equivalent to
$$
\bmatrix \pm \alpha_\ell & \\ & \pm \beta_\ell \\ & & \pm \sqrt{\alpha_\ell \beta_\ell} \\ & & & \mp \sqrt{\alpha_\ell \beta_\ell} \endbmatrix.
$$

An $L$-parameter is called unramified if it factors through $\Frob^\Z$. Thus unramified parameters are completely determined by its value on the Frobenius element. Unramified $L$-parameters are in bijection with unramified automorphic representations \cite[Proposition 1.12.1]{BR}. Hence this calculation completely determines the Asai transfer of unramified representations at unramified places.

\begin{remark}
Given $\pi = \otimes_v \pi_v$, the local Langlands correspondence gives the local Asai transfer $\As (\pi_v)$. If $\As (\pi) = \otimes_v \As(\pi_v)$, the global Asai transfer (i.e., the automorphy of $\As (\pi)$) is proved using the converse theorem. The analytic properties for certain $L$-functions needed to apply the converse theorem are proved using the Rankin--Selberg method (in \cite{Ram}) or the Langlands--Shahidi method (in \cite{Kris}).
\end{remark}

\medskip

\section{Brief overview of eigenvarieties due to Hansen}\label{sec:universal-eigenvarieties}

In this section, we first recall the notion of eigenvariety datum and the construction of an eigenvariety from such a datum. We then recall the construction of universal eigenvariety due to Hansen for certain reductive groups $G$. Finally, we recall the comparison theorem that allows us to construct rigid analytic maps between eigenvarieties. Our main reference for this section will be \cite{Hans}, and we adopt much of its notation.

\smallskip

\subsection{Eigenvariety data}
Let $p$ be an odd prime. An eigenvariety datum is defined as a tuple $\scrd = (\scrw, \scrz, \scrm, \T, \psi)$. We describe below each of the terms appearing in this definition.

The space $\scrw$ is a separated, reduced, relatively factorial rigid analytic space and is called the weight space. In our context, the space $\scrw$ will parametrize homomorphisms from the maximal torus of a reductive group $G$. The weight space contains as a dense subset a set of classical weights that support classical automorphic forms on $G$. 

Let $\mathbf{A}^1$ denote the rigid analytic affine line. The spectral variety $\scrz \subset \scrw \times \mathbf{A}^1$ is a Fredholm hypersurface, i.e., a closed immersion that is cut out by a Fredholm series. See \cite[Definition 4.1.1]{Hans} for the precise definition. Projection on the first coordinate induces a map $w : \scrz \to \scrw$ called the weight map. 

The ``overconvergent automorphic forms'' $\scrm$ will be a coherent  sheaf on $\scrz$. The sheaf is usually constructed from a suitable graded module $M^*$ of $p$-adic automorphic forms  or a complex whose cohomology yields $M^*$. The Hecke algebra $\T$ will be a commutative $\Q_p$-algebra equipped with an action $\psi: \T \to \End_{\calo_\scrz} (\scrm)$. The variety $\scrz$, in fact, will parametrize eigenvalues of an operator $U \in \T$ acting on $M^*$ or the complex whose cohomology is $M^*$. 

The following theorem gives us the eigenvariety associated to an eigenvariety datum. This follows from Buzzard's eigenvariety machine \cite{Buz}.

\begin{theorem}{\cite[Theorem 4.2.2]{Hans}}
Given an eigenvariety datum $\scrd$, there exists a separated rigid analytic space $\scrx$ together with a finite morphism $\pi: \scrx \to \scrz$, a morphism $w: \scrx \to \scrw$, an algebra homomorphism $\phi_\scrx : \T \to \calo (\scrx)$, and a coherent sheaf $\scrm^\dagger$ on $\scrx$ together with a canonical isomorphism $\scrm \cong \pi_* \scrm^\dagger$ compatible with the actions of $\T$ on $\scrm$ and $\scrm^\dagger$ (via $\psi$ and $\phi_\scrx$, respectively). The points of $\scrx$ lying over $z \in \scrz$ are in bijection with the generalized eigenspaces for the action of $\T$ on $\scrm(z)$ (the stalk at $z$). 
\end{theorem}

\smallskip
\subsection{Comparison theorem}

In this section we state the comparison theorem due to Hansen (which is a generalization of a comparison theorem of Chenevier \cite[Proposition 4.5]{Che}) enabling us to construct rigid analytic maps between eigenvarieties.

Let $\scrx$ denote the eigenvariety associated to the datum $\scrd = (\scrw, \scrz, \scrm, \T, \psi)$ as above. Let $\scrx^{\mathrm{red}}$ denote the nilreduction of $\scrx$. We denote by $\scrx^{\circ}$ (called the core  of $\scrx$), the union of $\dim\ \scrw$-dimensional irreducible components of $\scrx^{\mathrm{red}}$, viewed as a closed subspace of $\scrx$. If $\scrx^{\circ} \cong \scrx$, we say that the eigenvariety $\scrx$ is unmixed. The classical eigencurve and the Hilbert modular eigenvariety (associated to $\GL_{2}$ over a totally real field) are examples of eigenvarieties that are unmixed. On the other hand, the eigenvarieties constructed for $\GL_n$ for $n>2$ are not unmixed.

Inside the spectral variety $\scrz$, we define a subspace
$$
\scrz^{\circ} = \{ z \in \scrz \mid  \pi^{-1}(z) \in \scrx^\circ \}.
$$
In fact, $\scrz^{\circ}$ is naturally a union of irreducible components of $\scrz^{\mathrm{red}}$.

\begin{theorem}\label{compthm}\cite[Theorem 5.1.6]{Hans}
Suppose we have two eigenvariety datum
$$
\scrd_{i}= (\scrw_{i}, \scrz_{i}, \mathscr{M}_{i}, \T_{i}, \psi_{i}) \text{   for } i=1,2.
$$
Moreover, assume that we are given
\begin{itemize}
\item[(i)] a closed immersion between the weight spaces, $j: \scrw_{1} \hookrightarrow \scrw_{2}$;
\item[(ii)] a homomorphism between the Hecke algebras, $\sigma: \T_{2} \to \T_{1}$;
\item[(iii)] a Zariski accumulation dense set $\scrz_1^{\cl} \subset \scrz_{1}^{\circ}$ and an extension $j(\scrz_1^{\cl}) \subset \scrz_{2}$ such that for all $z \in \scrz_{1}^{\cl}$, there exists a inclusion $\mathscr{M}_{1}(z)^{ss} \hookrightarrow \mathscr{M}_{2}(j(z))^{ss}$ as $\T_{2}$-module. Here the notation ${^{ss}}$ denotes the semisimplification of a module.
\end{itemize}
Then, there exists a rigid analytic map between the eigenvarieties $i:\scrx_{1}^{\circ} \to \scrx_{2}$ such that the following two diagrams
\begin{center}
\begin{tabular}{ccc}
\begin{minipage}[t]{0.2\textwidth}
$$\xymatrix{
\scrx_{1}^{\circ} \ar[r]^{i}\ar[d]_{w_{1}} & \scrx_{2}\ar[d]^{w_{2}}   \\
\scrw_{1}\ar@{^{(}->}[r]_{j} & \scrw_{2} 
}$$
\end{minipage} &
\begin{minipage}[t]{0.1\textwidth}
\vspace{.25in}
$$\mathrm{and}$$
\end{minipage} &
\begin{minipage}[t]{0.2\textwidth}
$$\xymatrix{
 \calo(\scrx_{2}) \ar[r]^{i^{\ast}}& \calo(\scrx_{1}^{\circ}) \\
\T_{2}\ar[r]_{\sigma}\ar[u]^{\phi_{2}} & \T_{1} \ar[u]_{\phi_{1}}
}$$
\end{minipage}
\end{tabular}
\end{center}
commute.
\end{theorem}

\smallskip
\subsection{Universal eigenvariety} \label{eigenconstruct}
Throughout this section, let $G$ denote a connected reductive algebraic group over $\Q$ that is split at the prime $p$. Let $B, N, Z$ and $T$ denote a choice of a Borel subgroup, unipotent subgroup, the centre and maximal torus respectively. Let $I$ denote the Iwahori subgroup of $G(\Z_p)$ associated to the choice of $B$. In this subsection, we recall the definition and basic properties of the eigenvariety associated to $G$. 

The weight space $\scrw_G$ associated to $G$ is a rigid analytic space whose $\qpbar$ points are given by $\scrw_G (\qpbar) = \Hom_{\mathrm{cts}} (T(\Z_p), \qpbar^\times)$. For any open compact subgroup $K^p \subset G(\A_\Q^p)$, let $\scrw = \scrw_G (K^p)$ denote weight space of level $K^p$ which parametrizes continuous homomorphisms from the torus that are trivial on the closure of $Z(K^p I) \cap G(\Q) \subset T(\Z_p)$.  

A Hecke pair consists of a monoid $\Delta \subset G(\A_f)$ and a subgroup $K_f \subset \Delta$ such that $K_f$ and $\delta K_f \delta_f^{-1}$ are commensurable for all $\delta \in \Delta$. We will denote by $\T (\Delta, K_f)$ the $\Q_p$-algebra generated by double cosets $T_\delta = [K_f \delta K_f]$ under the convolution product. 

The algebra $\T$ is the eigenvariety data will be of the form
$$
\T = \T_G (K^p) = \scra^+_p \otimes \T^\mathrm{unr} (K^p),
$$
where $\scra^+_p$ is a certain subalgebra of the Iwahori Hecke algebra $\T (G(\Q_p), I)$ and the unramified Hecke algebra is a commutative algebra given by 
$$
\T^\mathrm{unr} (K^p) = {\bigotimes_{v \not \in S}}^\prime \T (G(\Q_v), K^p_v)
$$
for a finite set of primes $S$. In practice, $K^p_v$ will be a hyperspecial maximal subgroup of $G(\Q_v)$, which ensures the commutativity of the Hecke algebra. There is a subclass of operators in $\scra_p^+$ which are called controlling operators which play a crucial role in the construction of the eigenvariety.

We now proceed to describe the construction of the coherent sheaf $\scrm$ along with the action of the Hecke algebra. Let $\bar{B}$ and $\bar N$ respectively denote the opposite Borel and unipotent subgroup to $B$ and $N$. For any $s \in \Z_{\ge 0}$, define $\bar B^s = \{ b \in \bar B (\Z_{p}) \mid b \equiv 1 \in G (\Z_p/p^s \Z_p) \}$. Similarly, define ${T}^{s}= {T}(\Z_{p}) \cap \bar B^{s}$, $\bar N^{s}= \bar N (\Z_{p}) \cap \bar B^{s}$ and $I^{s} = I \cap \mathrm{Ker}\{G (\Z_{p}) \to {G}(\Z_p/ p^{s} \Z_p)   \}.$ Furthermore, we define $I^s_1 = \{ g\in I \mid g \mod p^s \in \bar N (\Z/p^s/\Z) \}$.

Let $\Omega \subset \scrw$ be an admissible affinoid open subset. For $\Omega$, the tautological character induced from $id \in \scrw (\scrw)$ is denoted by
$$
\chi_{\Omega}: {T}(\Z_{p}) \to \calo(\Omega)^{\times}.
$$
Let $s[\Omega]$ denote the smallest integer for which $\chi_{\Omega}|_{{T}^{s[\Omega]}}$ is analytic. For $s \geq s[\Omega]$, define
$$
A_{\Omega}^{s} =\Biggl\{ f: {I} \to \calo(\Omega) \bigg|  \begin{array}{c} 
f \text{ is analytic on each } {I}^{s} \text{-cosets and } \\
f(gtn) = \chi_{\Omega}(t) f(g) \text{ for all } n \in {N}(\Z_{p}), t \in {T}(\Z_{p}), g \in {I}
\end{array}    \Biggr\}.
$$
Via the map $f \mapsto f|_{\bar N^{1}}$, we identify $A_{\Omega}^{s}$ with the space of $s$-locally analytic $\calo(\Omega)$ valued functions on $\bar N^{1}$. Hence $A_{\Omega}^{s}$ is endowed with a Banach $\calo[\Omega]$-module structure. We have natural injective, compact transition maps $A_{\Omega}^{s} \to A_{\Omega}^{s+1}$ and taking direct limit with respect to these transition maps, we define
$$
\cala_{\Omega}= \varinjlim_{s} A_{\Omega}^{s}.
$$
The corresponding distribution space is defined as the continuous $\calo (\Omega)$-linear dual of $\cala_{\Omega}$,
$$
\mathcal{D}_{\Omega}= \{ \mu: \cala_{\Omega} \to \calo(\Omega)| \mu \text{ is } \calo(\Omega) \text{-linear and continuous} \}.
$$
If $\lambda \in \scrw$ is any point, we similarly define the modules $A_{\lambda}^{s}, \cala_{\lambda}$ and $\cald_{\lambda}$. 

Fix a controlling operator $U \in \scra_p^+$. For $\Omega$ as above, there exist complexes $C_\bullet (K^p I , \cala_\Omega)$ and $C^\bullet (K^p I, \cald_\Omega)$ that admit an extension $\tilde U$ of the controlling operator $U$.  We remark that the controlling operator acts as compact operator on $C_\bullet (K^p I , \cala_\Omega)$ and $C^\bullet (K^p I, \cald_\Omega)$. Let $f_\Omega (X)$ denote the Fredholm series associated to this action on $C_\bullet (K^p I, \cala_\Omega)$. Then these functions patch together to give Fredholm series $f \in \calo (\scrw)\{\{X\}\}$. The spectral variety $\scrz = \scrz_G$ is defined as the Fredholm hypersurface $\scrz_f$ associated to this series.

Given an affinoid open subset $\Omega \subset \scrw$ and $h \in \Q$, there is exists a corresponding affinoid open subset $\scrz_{\Omega, h} \subset \scrz$. We call this slope-adapted if $f_\Omega = f|_{\calo(\Omega)\{\{X\}\}}$ has a slopge-$\le h$ decomposition. The slope-adapted affinoids form an admissible open cover of $\scrz$. We also know that if $\Omega$ is slope adapted for $h$, then $C_\bullet (K^p I, \cald_\Omega)$ admits a slope-$\le h$ decomposition.  

Moreover, there exists a unique complex of coherent analytic sheaves, $\mathscr{K}^{\bullet}$,  on $\scrz$ such that $\mathscr{K}^\bullet (\scrz_{\Omega , h})= C^{\bullet}(K^{p} {I}, \cald_{\Omega})_{\le h}$ for any slope-adapted $\scrz_{\Omega,h}$. Taking cohomology of $\mathscr{K}^\bullet$, we get a graded sheaf $\mathscr{M}^{\ast}$ on $\scrz$, such that $\mathscr{M}^{\ast} (\scrz_{\Omega,h}) = H^{\ast}(K^{p} {I}, \cald_{\Omega})_{\le h}$. This sheaf comes equipped with a Hecke action, which we denote by $\psi: \T \to \End_{\calo_{\scrz}} (\mathscr{M}^*)$. Finally, we take our coherent sheaf $\scrm$ in the eigenvariety data to be the graded sheaf $\scrm^*$.

To summarise, the eigenvariety datum given by
$$
\scrd= (\scrw, \scrz, \scrm, \T, \psi)
$$
gives rise to the eigenvariety associated to $G$ via Buzzard's machinery. The eigenvariety $\scrx = \scrx (\scrd)$ thus constructed is a separated rigid analytic variety along with
\begin{enumerate}
\item[(i)] a finite morphism $\pi: \scrx \to \scrz$,
\item[(ii)] a morphism (weight map) $w: \scrx \to \scrw$,
\item[(iii)] an algebra homomorphism $\phi_{\scrx}: \T \to \calo(\scrx)$, and
\item[(iv)] a coherent sheaf $\mathscr{M}^{\dagger}$ on $\scrx$ together with a canonical isomorphism $\scrm \cong \pi_{\ast} \scrm^{\dagger}$ compatible with the action of $\T$.
\end{enumerate}

\smallskip
\subsection{Points on the eigenvariety and refinements.}
The points on the eigenvariety lying over $z \in \scrz$ are in bijection with the generalized eigenspaces for the action of $\T$ on $\mathscr{M} (z)$. 

\begin{definition}
A finite-slope eigenpacket of weight $\lambda \in \scrw (\qpbar)$ and level $K^{p}$ is a algebra homomorphism $\phi: \T \to \qpbar$ such that the space
$$
\{ v \in H^{\ast}(K^{p}{I}, \cald_{\lambda}) \otimes_{k_{\lambda}} \qpbar \mid T \cdot v = \phi(T)v \text{ for all } T \in \T \text{ and } \phi(U) \neq 0   \}
$$
is nonzero, that is, there exists a non-zero eigenvector acting by the homomorphism $\phi$. Here $U$ is the controlling operator fixed earlier.
\end{definition}

Another way to define finite-slope eigenpacket is as follows. We denote by ${\mathbb{T}}_{\lambda,h}(K^{p})$ the subalgebra of $\End_{k_{\lambda}}(H^{\ast}(K^{p} {I}, \cald_{\lambda})_{\le h})$ generated by the image of $\T \otimes_{\Q_{p}} k_{\lambda}$. We define the algebra ${\mathbb{T}}_{\lambda}(K^{p})= \varprojlim_{h} {\mathbb{T}}_{\lambda,h}(K^{p})$. Finite-slope eigenpackts of weight $\lambda$ and level $K^p$ can be identified with algebra homomorphisms $\phi: {\mathbb{T}}_{\lambda}(K^{p}) \to \qpbar$.

 Given any point $x \in \scrx (\qpbar)$ on the eigenvariety, we can naturally define an algebra homomorphism
$$
\xymatrix{
\phi_{{x}} : \T \ar[r]^-{\phi_{\scrx}}  & \calo(\scrx)  \ar[r] &  \calo_{\scrx,x} \ar@{->>}[r] & k_{x}}
$$
 called the eigenpacket parametrized by $x$. The points $x \in \scrx (\qpbar)$ lying over a given weight $\lambda \in \scrw (\qpbar)$ are in bijection with the finite-slope eigenpackets for ${G}$ of weight $\lambda$ and level $K^{p}$ obtained by sending $x \mapsto \phi_x$.
 
 Before we discuss the notion of refinements, we need to make explicit the Hecke operators that we consider at $p$. Let $\Phi^+$ denote the positive roots for the choice of the Borel subgroup $B$. Define the semi-group
 $$
 T^+ = \{ t \in T(\Q_p) \mid v_p (\alpha (t)) \le 0, \forall \alpha \in \Phi^+ \},
 $$
 and similarly
  $$
 T^{++} = \{ t \in T(\Q_p) \mid v_p (\alpha (t)) < 0, \forall \alpha \in \Phi^+ \}.
 $$
 For the monoid $\Delta = \Delta_p = I T^+ I$, we consider the Hecke operators $\scra_p^+ = \T (\Delta, I)$. Moreover, we define the space of Atkin-Lehner operators $\scra_p \subset \T (G(\Q_p), I)$ as the subalgebra of the Iwahori Hecke algebra generated by $U_t$ and $U_t^{-1}$, where $U_t = [ItI]$ denotes the double coset operators for $t \in T^+$. Furthermore, we take our controlling operators to be of the form $U_t$ where $t \in T^{++}$. See \cite[\S 2]{Hans} for further details.
 
Now suppose that locally $G$ is of the form $\GL_n/\Q_p$. Then, we take the following generators for $\scra_p^+$ and $\scra_p$. We denote by $U_{p,i}$ the element in $\scra_p^+$ given by the diagonal matrix
$$
(1, \dots, 1, p, \dots, p)
$$
where $p$ occurs $i$ times. We let $u_{p,i} = U_{p, i-1}^{-1} U_{p,i} \in \scra_p$.

Let $\pi_{p}$ be an unramified irreducible representation of ${\GL_n}(\Q_{p})$ defined over $L$. Denote by $r: \WD (\Q_p) \to \GL_n (L)$ the Weil--Deligne representation associated to $\pi_p$. Here $\WD(\Q_p)$ denotes the Weil--Deligne group of $\Q_p$. Let $\phi_{1},\dots,\phi_{n}$ be any ordering of eigenvalues of $r(\Frob_{p})$. This ordering of eigenvalues gives rise to a character $\chi$ of $\scra_{p}$ by the formula $\chi(u_{p,i})=p^{1-i} \phi_{i}$. The character $\chi$ is called a refinement of $\pi_{p}$. 

There exists a vector $0 \neq v \in \pi_{p}^{{I}}$, such that $\scra_{p}$ acts on $v$ by $\chi$.
If $\pi$ is a classical automorphic representation on ${G}$ such that $\pi_{p}$ is unramified and if $x \in \scrx (\qpbar)$ corresponds to $\pi$, then  we obtain a refinement of $\pi_{p}$ by considering $\phi_{{x}} |_{\scra_{p}}$. That is, each classical automorphic representation appears roughly $n!$ times in $\scrx$. Hence, we often denote classical points in the eigenvariety as a tuple $(\pi, \chi)$.

\smallskip
\subsection{Control theorem}
In this section, we define the notion of arithmetic weight and the space of classical automorphic forms above such weights. We then state the control theorem due to Ash--Stevens and Urban, which relates overconvergent automorphic forms with small slopes and classical automorphic forms. 

Let ${X}^{\ast}$ denote the integral weight lattice for $G$ and let  ${X}^{\ast}_{+} \subset X^\ast$ denote the subset of ${B}$-dominant weights. We call a weight $\lambda \in \scrw$ arithmetic if $\lambda = \lambda_1 \epsilon$ where  $\epsilon$ is a finite order character of ${T}(\Z_{p})$ and $\lambda_1 \in {X}^{\ast}$. Let $s[\epsilon]$ be the smallest integer such that $\epsilon$ is trivial on $T^{s[\epsilon]}$. Moreover, we say $\lambda$ is dominant arithmetic if $\lambda_1 \in {X}^{\ast}_{+}$. For the dominant weight $\lambda_1$, let $\mathscr{L}_{\lambda_1}$ denote the highest weight representation and let $f_{\lambda_1}$ denote the highest weight vector  associated to $\lambda_1$.

 The control theorem due to Ash--Stevens \cite{AS} and Urban \cite{Urb} which is a generalization of control theorem due to Stevens \cite{S} and Chenevier \cite{Che1,BC} relates the space of overconvergent automorphic forms $H^{\ast}(K^{p} {I}, \cald_{\lambda})$ and the classical automorphic forms $H^{\ast}(K^{p} I_1^{s}, \mathscr{L}_{\lambda_1})$, where $I^s_1 \subset I$ are the subgroups defined earlier.

Next, we note that, for $g \in {G}$ and $i \in {I}$ the function $f_{\lambda_1}(gi)$ defines an element in $\mathscr{L}_{\lambda_1} \otimes \cala_{\lambda}$ and pairing it with $\mu \in \cald_{\lambda}$ we obtain a map $i_{\lambda}: \cald_{\lambda} \to \mathscr{L}_{\lambda_1}$, which we symbolically write as 
$$
i_{\lambda}(\mu)(g) = \int f_{\lambda_1}(gi) \mu(i).
$$
Then $i_{\lambda}$ induces a morphism
$$
i_{\lambda}: H^{\ast}(K^{p} {I}, \cald_{\lambda}) \to H^{\ast}(K^{p} {I}_{1}^{s}, \mathscr{L}_{\lambda_1})
$$
for any $s \geq s[\epsilon]$. This map is an intertwining operator for the action of the Hecke algebra $\T$. For the definitions of the standard action of $\T$ on $H^{\ast}(K^{p} {I}, \cald_{\lambda})$ and the $\star$-action in weight $\lambda_1$ of $\T$ on $H^{\ast}(K^{p} {I}_{1}^{s}, \mathscr{L}_{\lambda_1})$, see \cite[\S 2.1]{Hans}.

Let $W$ denote the Weyl group of the torus $T$. For a controlling operator $U$, we call $h \in \Q$ small slope for a dominant arithmetic weight $\lambda = \lambda_1 \epsilon$ if 
$$
h < \inf_{w \in W \setminus \{1\}} v_{p}(w \cdot \lambda_1 (U)) - v_{p}(\lambda_1 (U)),
$$
where $v_p$ denotes the $p$-adic valuation. Recall that the Weyl group $W$ acts on weights by the rule $ w \cdot \mu = (\mu+\rho)^{w} - \rho$, where ${\rho} \in {X}^{\ast} \otimes_{\Z} \frac{1}{2} \Z$ denotes half the sum of positive roots with respect to $B$.

\begin{theorem}\label{controlthm} \cite[Theorem 3.2.5]{Hans}
For the controlling operator $U$, if $h$ is a small slope for a dominant arithmetic weight $\lambda = \lambda_1 \epsilon$, then there exists a natural isomorphism of Hecke modules
$$
H^{\ast}(K^{p} {I}, \cald_{\lambda})_{\le h} \cong H^{\ast}(K^{p} {I}_{1}^{s}, \mathscr{L}_{\lambda_1})_{\le h}^{{T}(\Z/p^{s}\Z)=\epsilon}.
$$
for all $s \ge s[\epsilon]$.
\end{theorem}

\medskip
\section{Eigenvarieties attached to $G_i$ and an auxiliary eigenvariety}\label{sec:relevant-eigenvarieties}

For the fixed odd prime $p$ that is unramified in $K$, let  $\mathcal{O}_{p} = \mathcal{O}_{K} \otimes \Z_{p}$. If $p$ splits  as $\gp \gp^c$ in $K$, then $\calo_p = \calo_{K, \gp} \times \calo_{K, \gp^c}$. If $\alpha \in \calo_p$, then write $\alpha = (\alpha_1, \alpha_2) \in \calo_{K, \gp} \times \calo_{K, \gp^c}$. On the other hand, if $p$ is inert in $K$, then $\calo_p$ is a degree $2$ extension of $\Z_p$. Given $\alpha \in \calo_p$, let $\alpha_1$ and $\alpha_2$ denote its Galois conjugates. 

\smallskip
\subsection{Weight spaces} 
We first start with a description of the weight spaces for $G_1 = \Res_{K/\Q} \mathbf{G}_1$. Let $\mathbf{B}_1$ and $\mathbf{T}_1$ denote the standard Borel subgroup and maximal torus in $\mathbf{G}_1$. Let $B_1 = \Res_{K/\Q} \mathbf{B}_1$ and $T_1 = \Res_{K/\Q} \mathbf{T}_1$ denote the corresponding subgroups in $G_1$. Let $Z_1$ denote the center of $G_1$ and let $I_1$ denote the Iwahori subgroup of $G_1$ (with respect to $B_1$). 

We take our level structure to be $K_{1} = \prod_v \GL_{2}(\calo_{K,v})$, where the product runs over all the non-archimedean places of $K$. For any $\Q_{p}$ affinoid algebra $A$, we have
$$
\scrw_{1}(A) := \scrw_{K_{1}^{p},G_1}(A)=\{ \kappa : T_1 (\Z_{p}) \to A^{\times} \mid \kappa \text{ is trivial on } \overline{ Z_{1}(\Q) \cap K_{1}^{p}I_1}    \}.
$$
Suppose that $n=(n_{1},n_{2})$ and $v=(v_{1},v_{2})$ are weights with $n+2v= mt$ as before. Let $\kappa: \mathbf{T}_1 (\calo_p) \to \qpbar^{\times}$ be the map, $\kappa(\alpha,\beta)= \prod_{i=1}^{2} \alpha_{i}^{n_{i} + v_i} \beta_{i}^{v_{i}}.$ The unit group $\calo_K^\times$ sits in $\mathbf{T}_1 (\calo_p)$ diagonally as $\gamma \mapsto (\gamma, \gamma)$.  Then $\kappa$ is in the weight space $\scrw_{1}$ if $\kappa$ is trivial on $\mathcal O_{K}^{\times}$ (not just on the totally positive units). If $k=n+2t$ and $w=v+n+t$, then Hilbert modular forms of weight $(k,w)$ (in the sense of Hida) has weight $\kappa$ in the weight space $\scrw_{1}$.

\begin{remark}
Note that we adopt a slightly different normalization, than usual, in our  definition of weight spaces. The usual normalization sends $(n,v)$ to $\kappa$ that maps $(\alpha, \beta) \mapsto \prod_{i=1}^2 \alpha_i^{n_i} \beta_i^{v_i}$ and the units $\calo_K^\times$ embeds as $\gamma \mapsto (\gamma, \gamma^2)$. See Buzzard~\cite{Buz}, for example.
\end{remark}

We now describe the weight space for the group $G_2 = \GL_4/\Q$. With our notation as before, let $B_2$ and $T_2$ denote the standard Borel and maximal torus of $G_2$. Let $Z_2$ denote the center of $G_2$ and let $I_2$ denote the Iwahori subgroup with respect to $B_2$.

We take our level structure to be $K_2 = \prod_{\ell} \GL_{4}(\Z_{\ell})$, where the product runs over all integer primes $\ell$. And the weight space $\scrw_{2}$ is defined similarly. For any $\Q_{p}$ affinoid algebra $A$, we have
$$
\scrw_{2}(A) := \scrw_{K_2^{p},G_2}(A)= \left\{ \chi: T_{2} (\Z_{p}) \to A^{\times}\ \biggl | \begin{array}{c} 
\chi \textrm{ is trivial on the closure} \\
\textrm{of } Z_{2}(\Q) \cap K_2^{p}I_2 
\end{array}   \right\}.
$$

\smallskip
\subsection{Hecke algebras}
Let $S$ denote the set of primes of $\Q$ which ramify in $K$ and let $\tilde{S}$ be the set of places of $K$ lying above primes in $S$. Let $S_{p}$ denotes the set of places in $K$ above $p$. 

For the group $G_1$, we define our unramified Hecke algebra as the commutative algebra
$$
\T^\mathrm{unr}_1 = \sideset{}{^\prime} {\bigotimes_{v \not \in S \cup \{ p \}}} \T (G_1 (\Q_v), G_1 (\Z_v)).
$$
Note that we are omitting Hecke operators at primes that are ramified in the quadratic extension $K$. At the prime $p$, we define a subrings $\scra_{1,p}^{+} \subset \scra_{1, p} \subset \T (G_1 (\Q_p), I_1)$, as before. Specifically, let $\Phi_1^+$ denote the set of positive roots for $B_1$. We define two semigroups $T_1^+$ and $T_1^{++}$ inside ${T_1}(\Q_{p})$ as
\begin{gather*}
T_1^{+}= \{ t \in {T_1}(\Q_{p}) \mid v_{p} (\alpha(t)) \le 0 \textrm{ for all } \alpha \in \Phi_1^{+} \} , \textrm{and} \\
T_1^{++}= \{ t \in {T_1}(\Q_{p}) \mid v_{p} (\alpha(t)) < 0 \text{ for all } \alpha \in \Phi_1^{+} \} .
\end{gather*}
Suppose for the moment that $p$ splits in $K$ and suppose that $t = (t_1, t_2) \in T_1 (\Q_p)$, where $t_1=\mathrm{diag}(p^{a_{1}}, p^{a_{2}})$ and $t_2 = \mathrm{diag}(p^{b_{1}}, p^{b_{2}})$. Then $t \in T_1^{+}$ if and only if $ a_{1} \leq a_{2}$ and $b_1 \leq b_2$. The same $t$ belongs to $T_1^{++}$ if and only if all the above inequalities are strict. We have similar conditions when $p$ is inert in $K$.

 For any $t \in T_1^{+} \cap {G_1}(\Z_{p})$, the double coset operators $U_{t}=[{I_1} t{I_1}]$ generate the algebra $\scra_{1, p}^{+}$.  The Atkin--Lehner algebra $\scra_{1, p}$ is a commutative subalgebra of $\T (G_1 (\Q_p), I_1)$ generated by $U_{t}$ and $U_{t}^{-1}$ with $t \in T_1^{+} \cap {G_1}(\Z_{p})$. We can naturally identify 
$$
\scra_{1, p}^{+} \cong \Q_{p} [T_1^{+} \cap {G_1}(\Z_{p})] \text{ and } \scra_{1, p} \cong \Q_{p}({T_1}(\Q_{p})/ {T_1}(\Z_{p})).
$$
We will call an operator $U_{t} \in \scra_{1, p}^{+}$ a controlling operator if $t \in T_1^{++}$. Finally, we define the Hecke algebra as
$$
\T_1 := \scra_{1, p}^{+} \otimes_{\Q_{p}} \T_1^{\mathrm{unr}} .
$$

We also view the unramified Hecke Algebra $\T_1^\mathrm{unr}$ as a product of local Hecke algebras as
$$
\T_1^\mathrm{unr} = \otimes^{\prime}_{\fr l \notin \tilde{S} \cup S_{p}} \T_{1, \fr l},
$$
where the local Hecke algebra $\T_{1, \fr l} = \T (\mathbf{G}_1 (K_{\fr l}), \mathbf{G}_1 (\calo_{K, \fr l}))$. For a place $\fr l$ of $K$ not in $\tilde{S} \cup S_{p}$, let  $\varpi_{\fr l}$ denote the uniformizer at $\fr l$. We denote by $T_{\fr l}$ and $S_{\fr l}$ the double coset operators $\Big[\GL_{2}(\calo_{K, \fr l}) \bmatrix \varpi_{\fr l} & \\ & 1 \endbmatrix \GL_{2}(\calo_{K, \fr l}) \Big]$ and $\Big[\GL_{2}(\calo_{K, \fr l}) \bmatrix \varpi_{\fr l} & \\ & \varpi_{\fr l} \endbmatrix \GL_{2}(\calo_{K, \fr l}) \Big]$ respectively. The operators $T_{\fr l}$ and $S_{\fr l}$ generates the local Hecke algebra $\T_{1, \fr l}$. Let $\gp \in S_{p}$. Let $\varpi_{\gp}$ denotes the  uniformizer at $\gp$. We denote by $U_{\gp}$ and $S_{\gp}$ the double coset operators $\Big[\mathbf{I}_1 \bmatrix 1 & \\ & \varpi_{\gp} \endbmatrix \mathbf{I}_1 \Big]$ and $\Big[\mathbf{I}_1 \bmatrix \varpi_{\gp} & \\ & \varpi_{\gp} \endbmatrix \mathbf{I}_1 \Big]$, respectively. Here $\mathbf{I}_1$ is the Iwahori subgroup with respect to the Borel subgroup $\mathbf{B}_1$. Then $U_{\gp}$ and $S_{\gp}$ for all $\gp \in S_{p}$ generates the algebra $\scra^+_{1, p}$.
 
 We take our controlling operator to be
 $$
 U_{p} = \prod_{v |p} U_{v}.
 $$
 
We now come to the group $G_2$, where our definitions are similar. The unramified Hecke algebra can be written as 
 $$
\T_2^\mathrm{unr} = {\otimes^{\prime}_{\ell \notin S \cup \{p \}}} \T_{2, \ell};
$$
 where the local Hecke algebra $\T_{2, \ell} = \T (G_2 (\Q_\ell), G_2 (\Z_\ell))$ is generated by operators $T_{\ell,i}$ corresponding to the double coset of matrix
 $$
 \mathrm{diag}( \underbrace{\ell, \dots, \ell}_{i} , 1, \dots, 1),
 $$
where $i = 1, \dots, 4$. We define $\scra^+_{2, p}$ and $\scra_{2,p}$ similarly. For each $i = 1,\dots, 4$, we denote  by $U_{p,i}$, the element of $\scra_{2, p}^{+}$ corresponding to the matrix
  $$
 \mathrm{diag}(1,\dots,1, \underbrace{p, \dots, p}_{i}).
 $$
 The operators $U_{p,i}$ generates the algebra $\scra_{2, p}^{+}$. The operators $u_{p,i}:=U_{p,i}U_{p,i-1}^{-1} \in \scra_{2, p}$ generates the Atkin--Lehner algebra $\scra_{2, p}$. Finally, we take 
 $$
 U_{p}= U_{p,1}U_{p,2}U_{p,3} \in \scra_{2, p}^{+}
 $$
 as our choice for the controlling operator. 
 
  For $i=1,2$, from the definition of weight space $\scrw_{i}$ and the Hecke algebra $\mathbb{T}_{i}$ and the choice of controlling operator $U_{p}$, we construct the eigenvariety $\scrx_{i}$ as in Section \ref{eigenconstruct}. This is the universal eigenvariety associated to the group $G_{i}$.

\smallskip
\subsection{The eigenvariety $\scrx$} 
We have previously constructed the eigenvarieties $\scrx_i$ associated to the groups $G_i$. In this section, we construct an auxiliary eigenvariety $\scrx$, which plays a role in the construction of the $p$-adic Asai transfer map when $p$ is inert.

Let $\scrd_2 = (\scrw_2, \scrz_2, \scrm_2, \T_2, \psi_2 )$ denote the eigenvariety datum associated to the eigenvariety $\scrx_2$. In order to construct $\scrx$, we only modify the Hecke algebra and keep the other objects the same as in $\scrx_2$.

Let $\tilde{T}_2^{+}$ denotes the subgroup $T_2^+$ whose elements are
$$
\tilde{T}_2^{+}=  \Biggl\{ \bmatrix p^{a_{1}} & \\ & p^{a_{2}} \\ & & p^{a_{3}} \\ & & & p^{a_{4}} \endbmatrix \bigg|  \begin{array}{c} 
a_{i} \in \mathbb{N} \cup \{0 \}, a_{1} \leq a_{2} \leq a_{3} \leq a_{4} \\
\text{ and } a_{3}-a_{2} \in 2\mathbb{N}\cup \{0 \}
\end{array}    \Biggr\}.
$$
Analogously, define
$$
\tilde{T}_2^{++}= \{ t \in \tilde{T}_2^{+} \mid a_{1} < a_{2} < a_{3} < a_{4} \}.
$$
Let $\tilde{\scra}_{2, p}^{+}$ (resp. $\tilde{\scra}_{2, p}$) denotes the  $\Q_{p}$ algebra generated by $U_{t}$ with $t \in \tilde{T}_2^{+}$ (resp. by $U_{t}, U_{t}^{-1}$ with $t \in \tilde{T}_2^{+}$). We call $U_{t} \in \tilde{\scra}_{2, p}^{+}$ a controlling operator if $t \in \tilde{T}_2^{++}$.

We denote by $\tilde{U}_{p,1},\tilde{U}_{p,2}, \tilde{U}_{p,3}$ and $\tilde{U}_{p,4}$ the double coset operators corresponding to the following matrices
$$
\begin{array}{cccc}
 \bmatrix 1 & \\ & 1 \\ & & 1 \\ & & & p \endbmatrix , &  \bmatrix 1 & \\ & 1 \\ & & p^2 \\ & & & p^2 \endbmatrix , &  \bmatrix 1 & \\ & p \\ & & p \\ & & & p \endbmatrix \text{ and} & \bmatrix p & \\ & p \\ & & p \\ & & & p \endbmatrix ,
\end{array}
$$
respectively. Then they generates the algebra $\tilde{\scra}_{2, p}^{+}$. We have a natural choice for the controlling operator
$$
\tilde{U}_{p}= \tilde{U}_{p,1} \tilde{U}_{p,2} \tilde{U}_{p,3} = \Bigg[ I_2 \bmatrix 1 & \\ & p \\ & & p^3 \\ & & & p^{4}  \endbmatrix I_2 \Bigg].
$$
Define $\tilde{u}_{p,1}= \tilde{U}_{p,1}, \tilde{u}_{p,2}= \tilde{U}_{p,2} (\tilde{U}_{p,1}^{-1})^{2}, \tilde{u}_{p,3}= \tilde{U}_{p,3} \tilde{U}_{p,2}^{-1} \tilde{U}_{p,1}$ and $ \tilde{u}_{p,4} = \tilde{U}_{p,4} \tilde{U}_{p,3}^{-1}$, then integral powers of $\tilde{u}_{p,i}$ generate the algebra $\tilde{\scra}_{2, p}$.

We define Hecke algebra $\tilde \T_2$ as
$$
\tilde \T_2 := \tilde{\scra}_{2,p}^{+} \otimes_{\Q_{p}} \T_2^\mathrm{unr},
$$
where $\T_2^\mathrm{unr}$ is the same as before.  Let $id: \tilde \T_2 \hookrightarrow \T_2$ denote the natural injection of Hecke algebras. Let $\scrd$ denotes the eigenvariety datum
$$
\scrd=(\scrw_{2}, \scrz_{2}, \scrm_{2}, \tilde \T_2, \psi_{2}|_{\tilde \T_2})
$$
and $\scrx$ denote the associated eigenvariety.

Let $\pi_p$ be an automorphic representation of $G_2 (\Q_p)$. We call a character $\tilde{\chi}$ of $\tilde{\scra}_{2, p}$ an accessible refinement of $\pi_{p}$ in $\scrx$, if there exists a character $\chi$ of $\scra_{2, p}$, such that $\tilde{\chi} = \chi|_{\tilde \scra_{2, p}}$ and $\chi$ is a refinement of $\pi_{p}$ appearing in $\scrx_2$.

Then applying the comparison theorem, Theorem~\ref{compthm}, to the maps
\begin{enumerate}
\item[(i)] $id: \scrw_{2} = \scrw_2$,
\item[(ii)] $id: \tilde \T_2 \hookrightarrow \T_2$
\end{enumerate}
and any very Zariski-dense set $\scrz_2^{cl} \subset \scrz_2$, we obtain a morphism 
$$
Q: \scrx_2^\circ \to \scrx.
$$ 
of rigid analytic spaces. The map $Q$ is finite, {\'e}tale and surjective. If $x,y \in \scrx_2^\circ$, then $Q(x)=Q(y)$ if $T_{\ell,i}(x)=T_{\ell,i}(y)$ for all $\ell \neq p$ and $i=1,\dots,4$, $U_{p,i}(x)=U_{p,i}(y)$ for $i=1,3,4$ and $U_{p,2}(x)= \pm U_{p,2}(y)$.

\medskip
\section{$p$-adic Asai transfer}\label{sec:p-adic-asai-transfer}

In this section, we construct a rigid analytic map between the eigenvarieties attached to $\GL_2/K$ and $\GL_4 /\Q$. This map is constructed using the comparison theorem due to Hansen described in the previous sections. In order to apply the theorem, we need to construct compatible maps at the level of weight spaces and Hecke algebras. We also need to define a Zariski accumulation dense set of {\em classical} points in the spectral variety $\mathscr Z_1 (\qpbar)$.

We first describe the map between the weight spaces. For the weight spaces $\scrw_1$ and $\scrw_2$ defined in the previous section, we construct the map $j: \scrw_{1} \hookrightarrow \scrw_{2}$ as follows. For a weight $\kappa \in \scrw_{1}$, define 
$$
j(\kappa)(t_{1},t_{2},t_{3},t_{4}) = (t_1 t_2)^{-1} \kappa(t_{1}t_{2},t_{3}t_{4},t_{1}t_{3},t_{2}t_{4}).
$$
We note that, if $\kappa(t_{1}, t_2 , t_3 , t_4 ) = t_{1}^{n_{1}+v_{1}} t_{2}^{v_{1}} t_{3}^{n_{2}+v_{2}}t_{4}^{v_{2}}$, then,
\begin{align*}
j(\kappa)(t_{1},t_{2},t_{3},t_{4}) & = t_{1}^{n_{1}+n_{2}+v_{1}+v_{2}-1} t_{2}^{n_{1}+v_{1}+v_{2}-1} t_{3}^{n_{2}+v_{1}+v_{2}}t_{4}^{v_{1}+v_{2}} \\
& = t_{1}^{m-1+\frac{n_{1}+n_{2}}{2}} t_{2}^{m-1+\frac{n_{1}-n_{2}}{2}} t_{3}^{m-\frac{n_{1}-n_{2}}{2}}t_{4}^{m-\frac{n_{1}+n_{2}}{2}},
\end{align*}
where, $m=n_{1}+2 v_{1} = n_2 + 2 v_2$.

The map between the Hecke algebras and the set of Zariski accumulation dense classical points depend on whether $p$ is split or inert. We will consider these two cases separately. In fact, when $p$ is inert we only construct a map to the auxiliary eigenvariety $\scrx$ attached to $\GL_4/\Q$. 

\smallskip
\subsection{The case where $p=\gp \gp^c$ is split in $K$}
We first construct the map between Hecke algebras attached to $G_1$ and $G_2$. The Hecke algebra $\T_2$ is generated by the elements $T_{\ell, i}$ (for $\ell \neq p$ and unramified in $K$) and $U_{p, i}$ for $i = 1,\dots, 4$. In $\T_1$, we also have the standard Hecke operators $T_{\fr l}$ and $S_{\fr l}$ for $\fr l$ away from $p$ and $U_{\gp}$ and $S_{\gp}$ for $\gp | p$. 

We define a map $\sigma^\pm : \T_2 \to \T_1$ as follows:

\begin{tabular}{cc}
\begin{minipage}[t]{0.5\textwidth}
\centering{\underline{When $\ell = \fr l \fr l^c$}}
$$\!\begin{aligned}
T_{\ell, 1} & \mapsto T_{\fr l} T_{\fr l^c} \\
T_{\ell, 2} & \mapsto T^2_{\fr l} S_{{\fr l}^c}  + S_{\fr l} T_{\fr l^c}^{2} - 2 \ell S_{\fr l}  S_{\fr l^c}  \\
T_{\ell, 3} & \mapsto \ell^{-1} T_{\fr l} S_{\fr l} T_{\fr l^c} S_{\fr l^c} \\
T_{\ell, 4} & \mapsto \ell^{-2} S^2_{\fr l} S^2_{\fr l^c}
\end{aligned}$$
\end{minipage} &
\begin{minipage}[t]{0.5\textwidth}
\centering{\underline{When $\ell$ is inert}}
$$\!\begin{aligned}
T_{\ell, 1} & \mapsto \pm T_\ell \\
T_{\ell, 2}  & \mapsto 0 \\
T_{\ell, 3}  & \mapsto \mp \ell^{-1} T_\ell S_\ell \\
T_{\ell, 4} & \mapsto -\ell^{-2} S^2_\ell \\
\end{aligned}$$
\end{minipage}
\end{tabular}

\begin{align*}
U_{p, 1} & \mapsto U_\gp U_{\gp^c} \\
U_{p, 2} & \mapsto U^2_\gp S_{\gp^c} \\
U_{p, 3} & \mapsto p^{-1} U_{\gp} S_{\gp} U_{\gp^c} S_{\gp^c} \\
U_{p, 4} & \mapsto p^{-2} S^2_{\gp} S^2_{\gp^c}.
\end{align*}

We now justify the definition of $\sigma^\pm$. Let $\pi$ be an automorphic representation coming from a Hilbert modular form $f$ of weight $\kappa = (n,v)$ as before. Let us choose a refinement of $f$ such that $U_{\gp}f =\alpha_{\gp}f $ and $U_{\gp^c}f =\alpha_{\gp^c}f $ and let $x \in \scrx_{1}$ be the corresponding point. We denote the eigenpacket associated to $x$ by $\phi_{\pi, \{\alpha_\gp, \alpha_{\gp^c} \}}$. Assume that $n_{1}>n_{2}$, then by our normalization $\As^\pm (\pi)$ is cohomological of weight $j(\kappa)$. For any refinement $\chi$ of $\As^\pm(\pi)$, for primes $\ell \nmid p$, $\As^\pm (\pi)_{\ell}$ is an unramified representation and the Hecke operators $T_{\ell,i}$ act on the spherical vector via the scalar $(\phi_{\pi, \{\alpha_\gp, \alpha_{\gp^c} \}})(\sigma(T_{\ell,i}))$.

First suppose that $\ell$ splits as $\fr l \fr l^c$ in $K$. Denote by $\alpha_{\fr l}$ and $\beta_{\fr l}$, the  $\Frob_{\fr {l}}$ eigenvalues of $f$. Then the characteristic polynomial for $\Frob_{\fr {l}}$ is given by
$$
X^{2}-T_{\mathfrak{l}}X+N(\mathfrak{l})S_{\mathfrak{l}} = (X-\alpha_{\fr l})(X-\beta_{\fr l}).
$$
Hence $T_{\fr {l}}$ acts by $\alpha_{\fr l}+\beta_{\fr l}$ and $N(\fr {l})S_{\fr {l}}$ acts by $\alpha_{\fr l}\beta_{\fr l}$. Similarly, the characteristic polynomial for $\Frob_{\fr l^c}$ is given by
$$
X^{2}-T_{\fr l^c}X+N(\fr l)S_{\fr l^c} = (X-\alpha_{\fr l^c})(X-\beta_{\fr l^c})
$$
where $\alpha_{\fr l^c}$ and $\beta_{\fr l^c}$ are the $\Frob_{\fr l^c}$ eigenvalues of $f$. On the other hand, characteristic polynomial for $\Frob_{\ell}$ corresponding to $\As^\pm (\pi)_\ell$ is given by
$$ 
X^{4}- T_{\ell ,1}X^{3} + \ell T_{\ell ,2}X^{2}- \ell^{3}T_{\ell ,3}X+ \ell^{6}T_{\ell ,4}.
$$
From our earlier calculation, we know that the $\Frob_{\ell}$ eigenvalues on $\As^{\pm}(\pi)_{\ell}$ are $\alpha_{\fr l}\alpha_{\fr l^c},\alpha_{\fr l}\beta_{\fr l^c},\beta_{\fr l}\alpha_{\fr l^c}$ and $\beta_{\fr l}\beta_{\fr l^c}$. Thus
$$
X^{4}- T_{\ell,1}X^{3} + \ell T_{\ell,2}X^{2}-\ell^{3}T_{\ell,3}X+\ell^{6}T_{\ell,4}
= (X-\alpha_{\fr l}\alpha_{\fr l^c})(X-\alpha_{\fr l}\beta_{\fr l^c})(X-\beta_{\fr l}\alpha_{\fr l^c})(X-\beta_{\fr l}\beta_{\fr l^c}).
$$
From this, we see that $T_{\ell,1}$ acts by the eigenvalue $\alpha_{\fr l}\alpha_{\fr l^c}+\alpha_{\fr l}\beta_{\fr l^c}+\beta_{\fr l}\alpha_{\fr l^c}+ \beta_{\fr l}\beta_{\fr l^c}=(\alpha_{\fr l}+\beta_{\fr l})(\alpha_{\fr l^c}+\beta_{\fr l^c}).$
On the other hand we know that $T_{\ell,1}$ acts by $\phi_{\pi, \{\alpha_\gp, \alpha_{\gp^c}\}}(\sigma(T_{\ell,1}))$. Hence our definition
$$
\sigma(T_{\ell,1})= T_{\mathfrak{l}}  T_{\mathfrak{l}^c}.
$$
The calculations for $T_{\ell, i}$ when $i = 2, 3,4$ are similar.

Now assume that $\ell$ is inert in $K$. Let $\alpha_\ell$ and $\beta_\ell$ denote the $\Frob_{\ell}$ eigenvalues of $f$. Then $\Frob_{\ell}$ eigenvalues on $\As^{\pm}(\pi)$ are given by $\pm \alpha_\ell, \pm \sqrt{\alpha_\ell \beta_\ell}, \mp \sqrt{\alpha_\ell \beta_\ell}$ and  $\pm\beta_\ell$. Thus
 $$
 X^{4}- T_{\ell,1}X^{3} + \ell T_{\ell,2}X^{2}-\ell^{3}T_{\ell,3}X+\ell^{6}T_{\ell,4} 
 =(X\mp \alpha_\ell)(X \mp \sqrt{\alpha_\ell \beta_\ell})(X \pm \sqrt{\alpha_\ell \beta_\ell})(X \mp \beta_\ell).
$$
We see that $T_{\ell,1}$ acts by $\pm(\alpha_\ell +\beta_\ell)$, $T_{\ell,2}$  by $0$, $T_{\ell,3}$ by $\mp \ell^{-3} \alpha_\ell \beta_\ell (\alpha_\ell + \beta_\ell)$ and $T_{\ell,4}$ acts by $-\ell^{-6}(\alpha_\ell \beta_\ell)^{2}$. Hence our definition of $\sigma^\pm$ above.

The following lemma will justify the definition of $\sigma^\pm$ for Hecke operators supported at $p$. Note that, $N(\gp)S_{\gp}$ (resp.\ $N(\gp^c)S_{\gp^c}$) acts via the eigenvalue $\alpha_{\gp}\beta_{\gp}$ (resp.\ $\alpha_{\gp^c}\beta_{\gp^c}$).

\begin{lemma}
The module $\As^{\pm}(\pi)_{p}^{I_2}$ contains a vector $v^\pm$ on which $\scra_{2, p}$ acts via the character associated to the tuple $(p^{-3}\beta_{\gp}\beta_{\gp^c}, p^{-2}\beta_{\gp}\alpha_{\gp^c},p^{-1}\alpha_{\gp}\beta_{\gp^c},\alpha_{\gp}\alpha_{\gp^c})$. In particular, $U_{p}$ acts via the scalar $p^{3m-1}\alpha_{\gp}^{4}\alpha_{\gp^c}^{2}$.
\end{lemma}

\begin{proof}
The proof of this lemma is similar to \cite[Lemma 5.5.2]{Hans}. We give a brief sketch here.
For this particular refinement, we see that $u_{p,1}=U_{p,1}$ acts by $\alpha_{\gp}\alpha_{\gp^c}$, $u_{p,2}=U_{p,2}U_{p,1}^{-1}$ acts by $p^{-1}\alpha_{\gp}\beta_{\gp^c}$, thus $U_{p,2}$ acts by $p^{-1}\alpha_{\gp}^{2}\alpha_{\gp^c}\beta_{\gp^c}$. Similarly $u_{p,3}=U_{p,3}U_{p,2}^{-1}$ acts by $p^{-2}\beta_{\gp}\alpha_{\gp^c}$ and hence $U_{p,3}$ acts by $p^{-3}\alpha_{\gp}^{2}\alpha_{\gp^c}^{2}\beta_{\gp}\beta_{\gp^c}$. Finally $u_{p,4}=U_{p,4}U_{p,3}^{-1}$ acts by $p^{-3}\beta_{\gp}\beta_{\gp^c}$ and $U_{p,4}$ acts by $p^{-6}\alpha_{\gp}^{2}\alpha_{\gp^c}^{2}\beta_{\gp}^{2}\beta_{\gp^c}^{2}$. 

From the characteristic polynomials of $U_{\mathfrak{p}}$ and $U_{\mathfrak{p}^c}$, we see that $\alpha_{\gp}\beta_{\gp}=\alpha_{\gp}\beta_{\gp^c}=p^{m+1}$. For the computation of the $U_{p}$ operator, notice that $U_{p}=U_{p,1}U_{p,2}U_{p,3}$ and hence acts via
\begin{align*}
(p^{-2} \beta_{\gp}\alpha_{\gp^c})(p^{-1}\alpha_{\gp}\beta_{\gp^c})^{2}(\alpha_{\gp}\alpha_{\gp^c})^{3} = p^{-4} (\alpha_{\gp}\beta_{\gp})(\alpha_{\gp^c}\beta_{\gp^c})^{2} \alpha_{\gp}^{4} \alpha_{\gp^c}^{2} 
= p^{3m-1} \alpha_{\gp}^{4}\alpha_{\gp^c}^{2}.
\end{align*}
This completes the proof of the lemma.
\end{proof}

The $\star$-action of $\scra_{2, p}$ on $\As^{\pm}(\pi)_p^{I_2}$ is the usual action, rescaled by $j(\kappa)(1,p,p^{2},p^{3})^{-1}$, where $j(\kappa)$ corresponds to the highest weight vector 
$$
\mu = \left(\frac{n_{1}+n_{2}}{2}+m-1, \frac{n_{1}-n_{2}}{2}+m-1, m-\frac{n_{1}-n_{2}}{2}, m-\frac{n_{1}+n_{2}}{2}\right).
$$
By our assumption that $n_1 > n_2$, the weight $\mu$ is a dominant integral weight. We compute that
\begin{align*}
\mu(1,p,p^2,p^3) &= p^{ \frac{n_{1}-n_{2}}{2}+m-1+2( m-\frac{n_{1}-n_{2}}{2})+3(m-\frac{n_{1}+n_{2}}{2})}\\ 
& = p^{6m-1-2n_{1}-n_{2}}\\
&= p^{3m-1+4v_{1}+2v_{2}}.
\end{align*}

Define the set of classical weights as 
$$
\scrw_1^{cl}= \{ (n,v) \in \scrw_1 \mid n_i, v_i \in \Z,\ n_{1} > n_{2} \ge 0 \textrm{ and } 2v_{1}+v_{2}=0 \}.$$
It is clear that $\scrw_1^{cl}$ is Zariski accumulation dense in $\scrw_{1}$.
Moreover, if $f$ has weight $(n,v) \in \scrw_1^{cl}$, then the eigenvalue of $\star$-action of $U_{p}$ on $\As^{\pm}(\pi)_p^{I_2}$ is $\alpha_{\gp}^{4}\alpha_{\gp^c}^{2}$.

Now we compute small slope $h$ for the weight $j(\kappa)=\mu$. By definition, we have
$$
h= \inf_{w \in S_{4}\setminus \{1\}} v_{p}((w \cdot \mu)(1,p,p^{2},p^{3})) - v_{p}(\mu(1,p,p^{2},p^{3})),
$$
where $ w \cdot \mu = (\mu+\rho)^{w} - \rho$.
We note that $\rho = (3,2,1,0)$ and 
\begin{multline*}
v_{p}((w \cdot \mu)(1,p,p^{2},p^{3})) - v_{p}(\mu(1,p,p^{2},p^{3})) 
 = (\mu_{w^{-1}(2)}- \mu_{2} + \rho_{w^{-1}(2)} - \rho_{2}) \\ + 2(\mu_{w^{-1}(3)}- \mu_{3} + \rho_{w^{-1}(3)} - \rho_{3})+3(\mu_{w^{-1}(4)}- \mu_{4} + \rho_{w^{-1}(4)} - \rho_{4}).
\end{multline*}
Hence $ v_{p}((w \cdot \mu)(1,p,p^{2},p^{3})) - v_{p}(\mu(1,p,p^{2},p^{3}))$ is a non-negative integer linear combination of 
\begin{gather*}
\mu_{1}-\mu_{2}+\rho_{1}-\rho_{2} = n_{2}+1, \\
\mu_{1}-\mu_{3}+\rho_{1}-\rho_{3} = n_{1}+1, \\
\mu_{1}-\mu_{4}+\rho_{1}-\rho_{4} =n_{1}+n_{2}+2, \\
\mu_{2}-\mu_{3}+\rho_{2}-\rho_{3} = n_{1}-n_{2}, \\
\mu_{2}-\mu_{4}+\rho_{2}-\rho_{4} = n_{1}+1, \\ 
\mu_{3}-\mu_{4}+\rho_{3}-\rho_{4} = n_{2}+ 1.
\end{gather*}
By taking $w = (1\ 2)$ and $w = (2\ 3)$, we see that 
$$
h=  \min \{ n_{1}-n_{2}, n_{2}+1\}.
$$

\begin{proposition}\label{prop1}
If $\alpha_{\gp}$ and $\alpha_{\gp^c}$ satisfy
$$
v_{p}(\alpha_{\gp}^{4}\alpha_{\gp^c}^{2})<  \min \{ n_{1}-n_{2}, n_{2}+1\},
$$
then $H^{\ast}(K_{2}^{p}I_{2},\mathcal{D}_{j(\kappa)})$ contains a nonzero vector $v^\pm$ such that every $T \in \T_2$ acts on $v^\pm$ through the scalar $\phi_{\pi,\{ \alpha_{\wp}, \alpha_{\wp^c} \}} (\sigma^{\pm}(T))$.
\end{proposition}

\begin{proof}
Since $n_{1}>n_{2}$, we have $\mu= j(\kappa)$ is dominant and $\As^{\pm}(\pi)$ is cohomological of weight $j(\kappa)$, see \cite{Ram}. By construction of the map $\sigma^{\pm}$ we see that $T_{\ell,i}$ acts by $\phi_{f,\{ \alpha_{\wp}, \alpha_{\wp^c} \}}(\sigma^{\pm}(T_{\ell,i}))$ on the line $\As^{\pm}(\pi)_{\ell}^{\GL_{4}(\Z_{\ell})}$. From the choice of our refinement and the $\star$-action of $U_{p}$ operator (for $G_{2}$), we see that $\As^{\pm}(\pi)_{p}^{I_2}$ contains a vector on which $U_{p}$ acts by $\alpha_{\gp}^{4}\alpha_{\gp^c}^{2}$. Finally we note that for $U_{p}$, any $h< h_{0}$ is a small slope for the dominant weight $\mu$. Applying Theorem \ref{controlthm} we obtain an isomorphism
$$
H^{\ast}(K_{2}^{p}I_2, \cald_{j(\kappa)})_{< h} \cong H^{\ast}(K_2^p I_2, \mathscr{L}_{j(\kappa)})_{< h},
$$
and the target contains a vector satisfying the claim of the theorem.
\end{proof}

The points on the spectral variety $\scrz_{1}(\qpbar)$ consists of set of tuples $(n_{1},n_{2},v_{1},v_{2}, \alpha^{-1})$ such that there exists a cuspidal overconvergent Hilbert eigenform $f$ of weight $(n_{1},n_{2},v_{1},v_{2})$ and $U_{\mathfrak{p}}^{4}U_{\mathfrak{p}^c}^{2} - \alpha$ annihilates $f$.  Let $\scrz_1^{cl}$ be the set of points in $\scrz_{1}$ of the form $(n_{1},n_{2},v_{1},v_{2}, \alpha^{-1})$ such that, $(n_{1},n_{2},v_{1},v_{2}) \in \scrw_1^{cl}$ and $\alpha < \min \{ n_{1}-n_{2}, n_{2}+1\}$. Then $\scrz_1^{cl}$ is a Zariski accumulation dense subset of $\scrz_{1}$. If $z=(\kappa, \alpha^{-1})$ and if $\underline{k}=(n_{1}+2,n_{2}+2)$ and $\underline{w}= (n_{1}+v_{1}+1, n_{2}+v_{2}+1)$, then $\mathscr{M}_{1}(z) \cong S_{\underline{k},\underline{w}}(K_{1}^{p}I_1)_{< \alpha}$. On the other hand, $\mathscr{M}_{2}(j(z)) \cong H^{\ast}(K_{2}^{p}I_2, \cald_{j(\kappa)})_{< \alpha} \cong H^{\ast}(K_{2}^{p}I_2, \mathscr{L}_{j(\kappa)})_{< \alpha}$. By classical Asai transfer we now have $\T_2$-equivariant inclusion
$\mathscr{M}_{1}(z) \hookrightarrow \mathscr{M}_{2}(j(z))$ for every $z \in \scrz_1^{cl}$. Applying the comparison theorem, Theorem \ref{compthm}, we obtain our desired map.

\begin{theorem}[$p$-adic Asai transfer: split case]
There exists a rigid analytic map
$$
\phi^{\pm} : \scrx_1 \to \scrx_2
$$
which sends the point $\pi, \{ \alpha_{\wp}, \alpha_{\wp}^c \}$ to the point $\As^{\pm}(\pi), \chi$, where $\chi$ is the refinement given by 
$$
\chi(u_{p,1}) = \alpha_{\gp}\alpha_{\gp^c}, \chi(u_{p,2})= p^{-1}\alpha_{\gp}\beta_{\gp^c}, \chi(u_{p,3})=  p^{-2}\beta_{\gp}\alpha_{\gp^c} \text{ and } \chi(u_{p,4})= p^{-3}\beta_{\gp}\beta_{\gp^c}.
$$
\end{theorem}

\smallskip
\subsection{The case where $p$ is inert in $K$}
Assume now that $p$ is inert in $K$. In this section, using the comparison theorem Theorem~\ref{compthm}, we construct a rigid analytic between $\scrx_1$ and $\scrx$, which at classical points interpolate `Asai transfer'.

The map at the level of the weight spaces remains the same. We define a map $\tilde \sigma^\pm : \tilde \T_2 \to \T_1$ as follows. On the unramified part of the Hecke algebra, the maps $\tilde \sigma^\pm$ agrees with $\sigma^\pm$. At $p$, we define $\tilde \sigma^\pm$ on $\tilde \scra^+_{2,p}$ by sending
\begin{align*}
\tilde{U}_{p, 1} & \mapsto \pm U_p \\
\tilde{U}_{p, 2} & \mapsto U^2_p S_{p} \\
\tilde{U}_{p, 3} & \mapsto \mp p^{-1} U_{p} S_{p}  \\
\tilde{U}_{p, 4} & \mapsto -p^{-2} S^2_{p}.
\end{align*}

The following lemma will justify the definition of $\tilde \sigma^\pm$ for Hecke operators supported at $p$. For a Hilbert modular form $f$ we choose the refinement such that, the $U_{p}$ eigenvalue of $f$ is $\alpha_{p}$. Note that, $N_{K/\Q}(p)S_{p}$ acts via the eigenvalue $\alpha_{p}\beta_{p}$.

\begin{lemma}\label{lem1}
The module $\As^{\pm}(\pi)_{p}^{I_2}$ contains a vector $v$ on which $\scra_{2, p}$ acts via the character $\chi$ associated to the tuple $(\pm p^{-3}\beta_{p}, \pm p^{-2} \sqrt{\alpha_{p}\beta_{p}} , \mp p^{-1}\sqrt{\alpha_{p}\beta_{p}}, \pm \alpha_{p})$. As a consequence, the character $\tilde{\chi}= \chi|_{\tilde{\scra}_{2,p}}$ of $\tilde{\scra}_{2,p}$ is a refinement associated to $\As^{\pm}(\pi)_{p}^{I_2}$ in $\scrx$. In particular, $\tilde{U}_{p}$ acts via the scalar $-p^{4m-1}\alpha_{p}^{4}$.
\end{lemma}

\begin{proof}
The proof of this lemma is similar to Lemma 4.4.1. We give a brief sketch here.

First, from the character $\chi$ of $\scra_p$, we get an explicit description of the character $\tilde{\chi}$ of $\tilde{\scra}_{2, p}$. We easily compute that
\begin{align*}
\tilde{\chi}(\tilde{u}_{p,1}) & =  \chi(u_{p,1})= \pm \alpha_{p} \\
\tilde{\chi}(\tilde{u}_{p,2}) & =  \chi(u_{p,2})^{2}= p^{-2} \alpha_{p}\beta_{p} \\
\tilde{\chi}(\tilde{u}_{p,3}) & =  \chi(u_{p,3}) \chi(u_{p,2})^{-1}= -p^{-1} \\
\tilde{\chi}(\tilde{u}_{p,4}) & =  \chi(u_{p,4})= \pm p^{-3} \beta_{p}.
\end{align*}
Observe that, $\tilde{U}_{p,1} = \tilde{u}_{p,1}$, $\tilde{U}_{p,2} = \tilde{u}_{p,2} \tilde{U}_{p,1}^{2}, \tilde{U}_{p,3}= \tilde{u}_{p,3} \tilde{U}_{p,2} \tilde{U}_{p,1}^{-1}$ and $\tilde{U}_{p,4} = \tilde{u}_{p,4} \tilde{U}_{p,3}$; hence, we get that $\tilde{U}_{p,1}, \tilde{U}_{p,2}, \tilde{U}_{p,3}$ and $\tilde{U}_{p,4}$ acts by $\pm \alpha_{p}, p^{-2}\alpha_{p}^{3}\beta_{p}, \mp p^{-3} \alpha_{p}^{2}\beta_{p}$ and $-p^{-6} \alpha_{p}^{2}\beta_{p}^{2}$ respectively.

From the characteristic polynomial of $U_{p}$, we see that $\alpha_{p}\beta_{p}=p^{2m+2}$. By definition, our controlling operator $\tilde{U}_{p}=\tilde{U}_{p,1}\tilde{U}_{p,2}\tilde{U}_{p,3}$. We compute the action of $\tilde{U}_{p}$ as
\begin{align*}
(\pm \alpha_{p})(p^{-2}\alpha_{p}^{3}\beta_{p})(\mp p^{-3} \alpha_{p}^{2}\beta_{p}) = -p^{-5} (\alpha_{p}\beta_{p})^{2} \alpha_{p}^{4}
= - p^{4m-1} \alpha_{p}^{4}.
\end{align*}
This completes the proof of the lemma.
\end{proof}

The $\star$-action of $\tilde{\scra}_{2,p}^{+}$ on $\As^{\pm}(\pi)_p^{I_2}$ is the usual action, rescaled by $j(\kappa)(1,p,p^{3},p^{4})^{-1}$, where $j(\kappa)$ corresponds to the highest weight vector 
$$
\mu = \left( \frac{n_{1}+n_{2}}{2}+m-1, \frac{n_{1}-n_{2}}{2}+m-1, m-\frac{n_{1}-n_{2}}{2}, m-\frac{n_{1}+n_{2}}{2} \right).
$$
By our assumption that $n_1 > n_2$, the weight $\mu$ is a dominant integral weight. We compute that
\begin{align*}
\mu(1,p,p^2,p^3) &= p^{ \frac{n_{1}-n_{2}}{2}+m-1+3 ( m-\frac{n_{1}-n_{2}}{2})+ 4(m-\frac{n_{1}+n_{2}}{2})}\\ 
& = p^{8m-1-3n_{1}-n_{2}}\\
&= p^{4m-1+6v_{1}+2v_{2}}.
\end{align*}

Define the set of classical weights as 
$$
\scrw_1^{cl}= \{ (n,v) \in \scrw_1 \mid n_i, v_i \in \Z,\ n_{1} > n_{2} \ge 0 \textrm{ and } 3v_{1}+v_{2}=0 \}.$$
It is clear that $\scrw_1^{cl}$ is Zariski accumulation dense in $\scrw_{1}$.
Moreover, if $f$ has weight $(n,v) \in \scrw_1^{cl}$, then the eigenvalue of $\star$-action of $U_{p}$ on $\As^{\pm}(\pi)_p^{I_2}$ is $-\alpha_{p}^{4}$.

Now we compute small slope $h$ for the weight $j(\kappa)=\mu$. By definition, we have
$$
h= \inf_{w \in S_{4}\setminus \{1\}} v_{p}((w \cdot \mu)(1,p,p^{2},p^{3})) - v_{p}(\mu(1,p,p^{2},p^{3})).
$$
As in the split case, we note that
\begin{multline*}
v_{p}((w \cdot \mu)(1,p,p^{2},p^{3})) - v_{p}(\mu(1,p,p^{2},p^{3})) 
 = (\mu_{w^{-1}(2)}- \mu_{2} + \rho_{w^{-1}(2)} - \rho_{2}) \\ + 3 (\mu_{w^{-1}(3)}- \mu_{3} + \rho_{w^{-1}(3)} - \rho_{3})+ 4(\mu_{w^{-1}(4)}- \mu_{4} + \rho_{w^{-1}(4)} - \rho_{4}).
\end{multline*}
By taking $w = (1\ 2)$ and $w = (2\ 3)$, we see that 
$$
h=  \min \{  n_{2}+1, 2(n_{1}-n_{2}) \}.
$$

\begin{proposition}
If
$$
v_{p}(\alpha_{p})<  \min \left\{ \frac{n_{2}+1}{4}, \frac{n_{1}-n_{2}}{2} \right\},
$$
then $H^{\ast}(K_{2}^{p} I_2,\mathcal{D}_{j(\kappa)})$ contains a nonzero vector $v^{\pm}$ such that every $T \in \tilde \T_2$ acts on $v^{\pm}$ through the scalar $\phi_{f,\alpha_p}(\tilde \sigma^{\pm}(T))$.
\end{proposition}

\begin{proof}
The proof follows from the control theorem, Theorem~\ref{controlthm}, exactly as in Proposition \ref{prop1}.
\end{proof}

Let $\scrz_{1}(\qpbar)$ consists of set of tuples $(n_{1},n_{2},v_{1},v_{2}, \alpha^{-1})$ such that there exists a cuspidal overconvergent Hilbert eigenform $f$ of weight $(n_{1},n_{2},v_{1},v_{2})$ and $U_{p}^{4} + \alpha$ annihilates $f$. Let $\scrz_1^{cl}$ be the set of points in $\scrz_{1}$ of the form $(n_{1},n_{2},v_{1},v_{2}, \alpha^{-1})$ such that, $(n_{1},n_{2},v_{1},v_{2}) \in \scrw_1^{cl}$ and $\alpha$ satisfies $v_{p}(\alpha) < \min \{ \frac{n_{1}-n_{2}}{2}, \frac{n_{2}+1}{4}\}$. Then $\scrz_1^{cl}$ is a Zariski accumulation dense subset of $\scrz_{1}$. By classical Asai transfer we now have $\tilde \T_2$-equivariant inclusion $\mathscr{M}_{1}(z) \hookrightarrow \mathscr{M}_{2}(j(z))$ for every $z \in \scrz_1^{cl}$. Applying the comparison theorem, Theorem \ref{compthm}, we obtain our desired map.

\begin{theorem}[$p$-adic Asai transfer: inert case]
There exists a rigid analytic map
$$
\phi^{\pm} : \scrx_1 \to \scrx
$$
which sends the point $(f, \alpha_p)$ to the point $(\As^{\pm}(\pi), \tilde{\chi})$ where $\tilde{\chi}$ is the refinement given in Lemma \ref{lem1}.

\begin{remark}
It remains an interesting question to determine if the map $\phi^\pm$ can be lifted to make the following diagram commute.
$$
\xymatrix@1{
 & \scrx^\circ_2  \ar[d]^{Q} \\
 \scrx_1 \ar[r]_{\phi^\pm} \ar@{-->}[ur]^{\exists \tilde \phi^\pm} & \scrx. 
 }
 $$
\end{remark}

\end{theorem}

\end{document}